\documentclass[11pt,twoside,a4paper]{article}

\newcounter{licznik}[section]

\newtheorem{Definition}[licznik]{Definition}
\newtheorem{Proposition}[licznik]{Proposition}
\newtheorem{Theorem}[licznik]{Theorem}
\newtheorem{Lemma}[licznik]{Lemma}
\newtheorem{Corollary}[licznik]{Corollary}
\newtheorem{Remark}[licznik]{Remark}

\usepackage{hyperref, sansmath}

\newcommand{\fla}{\frac{\lambda}{\alpha}}

\newcommand{\atl}{\frac \alpha {2\lambda}}

\newcommand{\ftla}{\frac{2\lambda}{\alpha}}
\def\flas{\frac \lambda {\alpha^2}}

\newcommand{\R}{\mathbb{R}}
\newcommand{\F}{\mathbb{F}}

\def\hye{\hat y_3}
\def\fotd{\frac 1 {2\d}}

\def\a{\alpha}
\def\la{\lambda}
\def\sta{\sqrt{2\alpha}}
\def\cL{\mathcal{L}}
\def\cH{\mathcal{H}}

\def\lb{\left(}
\def\rb{\right)}

\def\hyo{\hat y_1}
\def\hyt{\hat y_2}

\newtheorem{theorem}{Theorem}[section]

\usepackage{amsfonts,amssymb,amsmath,latexsym,epsfig,verbatim, xcolor, hyperref}

{}
\newcommand{\pd}[2]{\frac{\partial#1}{\partial#2}}
\def\beas{\begin{eqnarray*}}
\def\eeas{\end{eqnarray*}}

\def\cF{\mathcal{F}}

\def\d{\delta}

\def\t{\tau}

\def\cS{\mathcal{S}}
\def\cA{\mathcal{A}}

\def\tx{\tau}

\def\RR{\mathbb R}
\def\EE{\mathsf E}
\def\PP{\mathsf P}

\def\cF{{\cal F}}
\def\cA{{\cal A}}

\def\cS{{\cal S}}

\title{A geometric answer to an open question of singular control with stopping}
\author{John Moriarty\thanks{School of Mathematical Sciences, Queen Mary University of London, Mile End Road,
London E1 4NS, United Kingdom; \texttt{j.moriarty@qmul.ac.uk}}}
\begin{document}
\maketitle

\begin{abstract}
We solve a problem of singular stochastic control with discretionary stopping, suggested as an interesting open problem by Karatzas, Ocone, Wang and Zervos (2000), by providing suitable candidates for the moving boundaries in an unsolved parameter range. We proceed by identifying an optimal stopping problem with similar variational inequalities and inspecting its parameter-dependent geometry (in a sense going back to Dynkin (1965)), which reveals a discontinuity not previously exploited. We thus highlight the potential importance of this geometric information in both singular control and parameter-dependent optimal stopping.
\end{abstract}

\section{Introduction}
\label{sec:intro}

Consider the parameter-dependent family of optimal stopping problems
\begin{equation}\label{eq:osfamily}
V(x;c) := \inf_{\tau \in \cS}\EE\left[e^{-r\tau} \ell(X_\tau;c) \right], \qquad x \in \R,
\end{equation}
where $x \mapsto V(x;c)$ is the {\em value function}, $(X_t)_{t \geq 0}$ is a one-dimensional regular diffusion started at $x$, $\cS$ is the set of stopping times, 
and $(x \mapsto \ell(x;c))_{c \geq 0}$ is a family of {\em obstacles} indexed by a parameter $c$. It is well known that the optimal stopping time may be identified as the first hitting time by $X$ of the set $\cS_c$ on which the obstacle and the value function coincide (see, for example, \cite{Peskir2006}). Parameter-dependent optimal stopping problems for one-dimensional diffusions were recently studied in \cite{Bank2010} where the obstacle depends linearly on the parameter and a convenient representation for the optimal stopping set is derived, as the level set of an auxiliary process. In contrast, our optimal stopping problems come from an open problem of singular stochastic control with discretionary stopping suggested by Karatzas, Ocone, Wang and Zervos \cite{KOWZ00} (which we refer to below as the {\em control problem}) and the parameter dependence is nonlinear. By studying the parameter-dependent geometry in a sense going back to Dynkin \cite{Dynkin} we show that a boundary of $\cS_c$ is discontinuous in the parameter, which enables the control problem to be solved in a previously open case. In this way we highlight the potential importance of the information contained in the geometry of parameter-dependent optimal stopping problems, albeit at the cost of potentially more complex stopping regions than the convenient representations obtained in \cite{Bank2010}.

In associating the control problem to an optimal stopping problem having similar variational inequalities we place it in the context of a series of papers by authors including Karatzas and Shreve \cite{KaratzasShreve84}, \cite{KaratzasShreve85}  and El Karoui and Karatzas \cite{ElKK88}, \cite{ElKK91}. Beginning with \cite{Baldursson97} more explicit approaches have also been developed for this connection, and here we provide explicit analytical expressions for the control problem value function.

\begin{figure}[htbp]
\begin{center}
\includegraphics[scale=1.55]{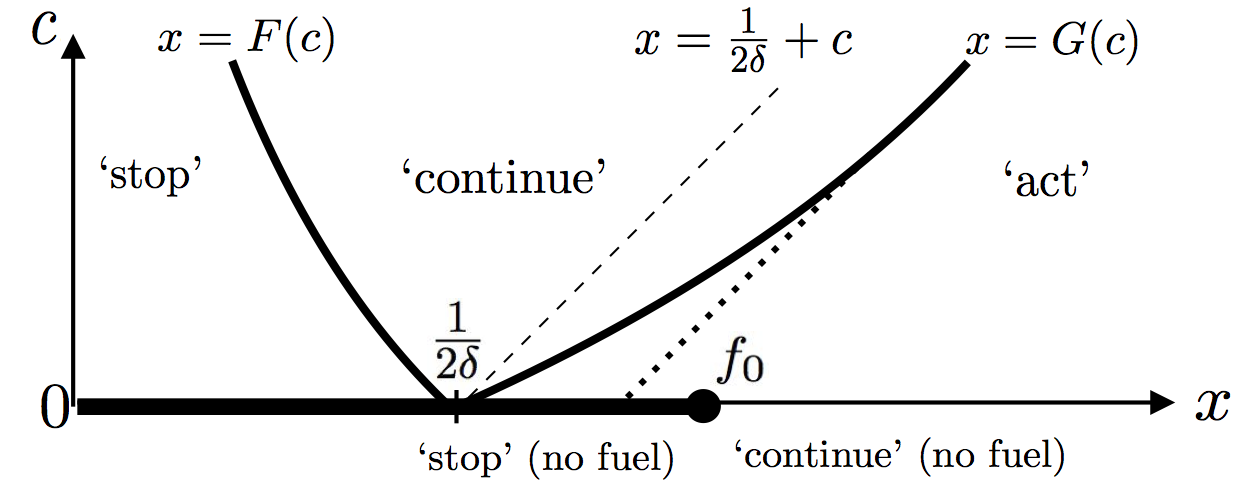}
\includegraphics[scale=1.5]{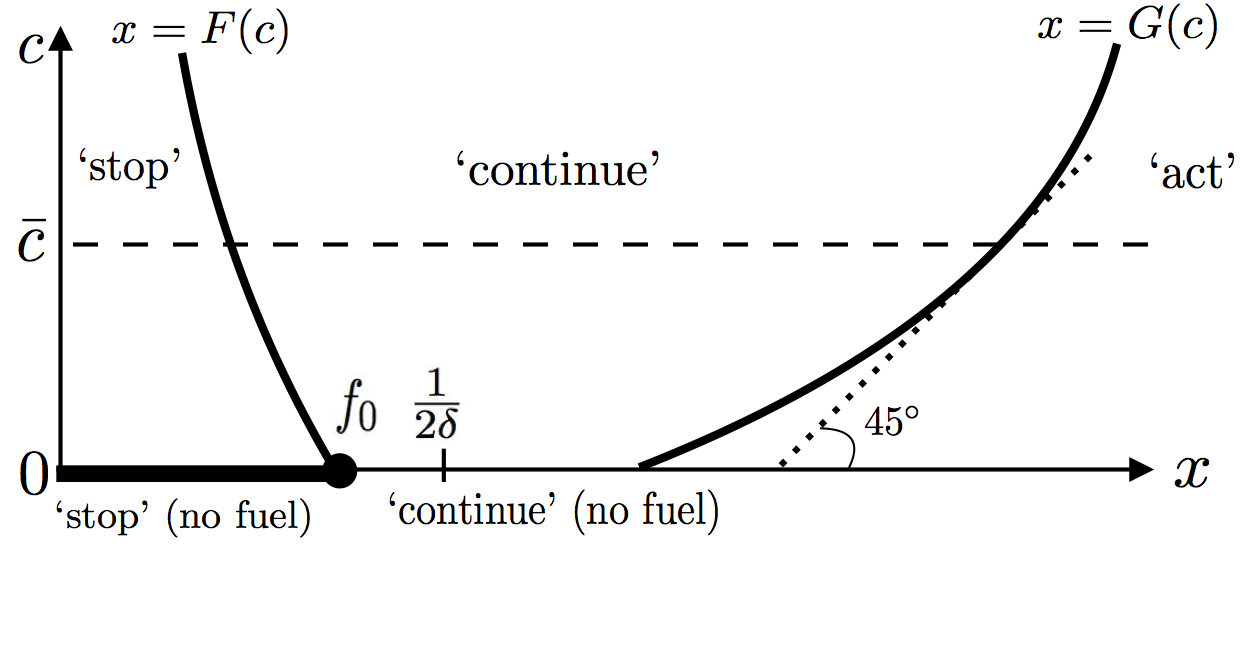}
\caption{Moving boundaries of the control problem for two different parameter choices. The direction of control is south-west (dotted line). Top (obtained below): here $f_0 > \fotd$ and $\la \in (\la^\dagger, \a \d)$. The stopping/continuation boundary is discontinuous at $c = 0$ and repulsion at the right boundary $G$ may be followed either by immediate stopping or by continuation, then stopping at a later time. Bottom (obtained in \cite{KOWZ00}): here $f_0 \leq \fotd$ and $\la \in (\la_*,\la^*]$.
The right boundary $G$ is reflecting for $c>\bar c$ and repelling for $c \in (0,\bar c]$ and repulsion is almost surely followed by continuation, then stopping.  
}
\label{fig:boundaries}
\end{center}
\end{figure}
\subsection{The control problem}

The control problem, which is formally defined in Section \ref{sec:method}, has state space $\{(x,c):x \in \R, c \geq 0\}$ and $x$ represents a {\em position} while $c$ is the {\em fuel level}. One particular solution obtained in \cite{KOWZ00} is illustrated in the bottom panel of Figure \ref{fig:boundaries}. This figure divides the portion $[0,\infty)^2$ of the state space (this restriction is justified in Remark \ref{rem:halfspace} below) into `stop', `continue' and `act' regions in which respectively it is optimal to end the problem, to do nothing, or to adjust the position by consuming fuel. The left moving boundary (between the `stop' and `continue' regions) is referred to as {\em absorbing}. Because of the south-west direction in which the control acts (dotted diagonal line), the right moving boundary $G$ (between the `continue' and `act' regions) is {\em reflecting} for $c>\bar c$ and {\em repelling} for $c \leq \bar c$. 
In particular at the repelling boundary $G(c)$ $(c \in (0, \bar c])$ all available fuel is expended. 


The rest of the paper is organised as follows. After providing background (Section \ref{sec:method}), 
 we identify the family of optimal stopping problems associated to the control problem in Section \ref{sec:methods}. In Sections \ref{sec:nofuel} and \ref{sec:cpos} these problems are solved geometrically, their parameter-dependent boundaries being analysed in Section \ref{sec:sol1fam}. The optimality of these boundaries for the control problem 
  is verified in Section \ref{Sec:verification}.

\section{Background}
\label{sec:method}

For convenience we recall here the setup of the control problem. Consider a probability space $(\Omega, \cF, \PP)$ equipped with a filtration $\mathbb{F}=\{\cF_t, 0 \leq t < \infty \}$ satisfying the usual conditions of right continuity and augmentation by null sets, and let $\cS$ be the set of all $\F$-stopping times. Denote by $\cA$ the class of $\F$-adapted, right-continuous processes $\xi=\{\xi_t, 0 \leq t < \infty\}$ with finite total variation on any compact interval and with $\xi_{0-}=0$. A process $\xi \in \cA$ is considered in its minimal decomposition 
$$\xi_t=\xi^+_t-\xi^-_t, \qquad t \in [0,\infty),$$
as the difference of two non-decreasing processes $\xi^\pm \in \cA$, so that its total variation on the interval $[0,t]$ is 
$$\check \xi_t = \xi_t^+ + \xi_t^-, \qquad t \in [0,\infty],$$
and for $c \in [0,\infty)$ we write 
$$\mathcal{A}(c) = \{\xi \in \cA:\check \xi_\infty \leq c, \text{ a.s.}\}.$$ 
We assume also that $(\Omega, \cF, \PP)$ supports the $\F$-adapted Wiener process $W=\{W_t, 0 \leq t < \infty\}$. Given an {\em initial position} $x \in \R$, {\em initial fuel level} $c \geq 0$ and {\em control process} $\xi \in \cA(c)$, we define the {\em state process}
\begin{eqnarray}
\label{statevariable}
X_t&=&x + W_t, \\
(Y_t, C_t) &=&(X_t + \xi_t, c-\check \xi_t),
\end{eqnarray}
for $t \geq 0$.
The control problem is then defined by the value function
\begin{eqnarray}
Q(x;c) =  \inf_{\xi \in \cA(c),\t \in \cS} 
\EE\left[ \int_0^{\tx}e^{-\alpha t}\lambda Y^2_t dt + \int_{[0,\tx]} e^{-\alpha t} d \check \xi_t + e^{-\alpha\t}\delta Y^2_{\t}\cdot 1_{\{\tau < \infty\}}
\right], \label{eq:defV}
\end{eqnarray}
where $\la>0$, $\a>0$ and $\d > 0$. 

In \cite{KOWZ00} further details on the interpretation and context of the control problem are given. Solutions are obtained by constructing evenly symmetric candidate value functions $\tilde Q(\cdot;c)$ using the associated Hamilton-Jacobi-Bellman equation and applying a verification theorem. This procedure is carried out in the cases $\la \geq \a \d$,  $\la \in (0, \la_*]$ and $\la \in (\la_*, \la^*]$, where
\[
\la_* \leq \la^* = \frac{\a \d}{1 + \frac{\d/\a}{(1/4\d)+(1/\sta)}} < \a \d.
\]
It is convenient to note that in the open case $\la \in (\la^*,\a\d)$ we have $f_0 > \frac 1 {2\d}$, while $f_0 \leq \frac 1 {2\d}$ in the previously solved case $\la \in (0,\la^*]$.
Here $f_0$ is the unique positive solution of the equation $\rho(f_0)=0$ where
\begin{eqnarray}
\rho(x)&:=&x^2+\frac{ 2 x}{\sqrt{2\a}} -\frac{\la / \a}{\a\d-\la}, \label{eq:defrho}
\end{eqnarray}
so that
\begin{eqnarray}
f_0&=&f_0(\la)=\frac 1 \sta \left(\sqrt{\frac{\a \d + \la}{\a \d - \la}}-1 \label{eq:fodef}
\right)>0.
\end{eqnarray} 
The quantity $f_0$ is also the free boundary for the problem without fuel (that is, when $c=0$) and it plays a key role below.
\begin{Remark}\label{rem:halfspace} Below we simplify the control problem as follows: (i) Since the abovementioned verification procedure for candidate value functions which are even in $x$ will be applied (in Section \ref{Sec:verification}), we consider the control problem only on the `half' state space $\{(x,c):x, c \geq 0\}$ as in Figure \ref{fig:boundaries}; (ii) it follows from a simple comparison argument that we need consider only 
monotone controls by taking $\xi^+ \equiv 0$ so that
$(Y_t, C_t) =(X_t - \xi^-_t, c- \xi^-_t)$
 for $t \geq 0$; (iii) we will consider only stopping times which take almost surely finite values (this will be justified by the verification argument).
\end{Remark}

\section{A related family of optimal stopping problems}
\label{sec:methods}

In order to connect the control problem \eqref{eq:defV} to a family of optimal stopping problems, we first discuss a known solution depicted in the bottom panel of Figure \ref{fig:boundaries}. In this case, when the initial state $(x,c)$ lies in the `continue' region  and $c \in [0,\bar c]$, the optimal policy is to wait until either the problem ends with absorption at the left boundary or until all fuel is expended instantaneously (upon reaching the right boundary). This suggests that for such initial values $(x,c)$ the control problem is in fact one of optimally stopping either at the left or right boundary, where the costs incurred at the right boundary are derived using the principle of optimality and the value function with no fuel.
In particular since the corresponding optimal stopping rule then has no direct dependence on the optimal policy for intermediate fuel levels $\tilde c \in (0,c)$, the family of optimal stopping problems parameterised by $c \in [0,\bar c]$ may be solved independently of each other.

When the initial fuel level is $c=0$ we necessarily have $\check \xi \equiv 0$ almost surely. In this case $Y \equiv X$ in \eqref{eq:defV}, (singular) control cannot be exercised and the only available intervention is stopping at the time $\tau$. The control problem then reduces to the following optimal stopping problem:
\begin{eqnarray}
\tilde V_0(x) &:=& \inf_{\t \in \cS} \EE\bigg[\int_0^{\t}e^{-\alpha t}\lambda X^2_t dt + e^{-\alpha\t} \delta X^2_{\t}\bigg]. \label{eq:defq0}
\end{eqnarray}
More generally for $(x,c)$ in the `continue' region with $c \in [0,\bar c]$ as above, in this solution the fuel level remains at the constant value $c$ until the first hitting time of the right boundary. The control policy is therefore $\xi \equiv 0$ in the related optimal stopping problem which, by \eqref{eq:defV} and the principle of optimality, is
\begin{eqnarray}
\tilde V(x;c) &:=& \inf_{\t \in \cS} \EE\bigg[\int_0^{\t}e^{-\alpha t}\lambda X^2_t dt + e^{-\alpha\t}\min(c+Q_0(X_\t-c), \delta X^2_{\t})\bigg]  \label{eq:defqc}
\end{eqnarray}
where $Q_0(x):=Q(x;0)$, $x \in \R$. This formulation as a family of optimal stopping problems is consistent with the candidate solution constructed and verified in Section 10 of \cite{KOWZ00}.

In the open case $f_0 > \fotd$ we will apply our optimal stopping heuristic to construct a candidate solution in an open parameter range (as discussed below in Remark \ref{rem:heurbreak}, the optimal stopping heuristic will break down in the remainder of the open range, in which the control problem remains as an interesting open problem). The moving boundaries of this new solution are illustrated in the top panel of Figure \ref{fig:boundaries}. We also confirm that the value function \eqref{eq:defV} is continuously differentiable across these boundaries, answering a question raised in \cite{KOWZ00}.

\subsection{Solution via obstacle geometry}
We first rewrite problem \eqref{eq:defq0} for $c=0$, and the problems \eqref{eq:defqc} for $c \in (0,\bar c)$, as a family of  problems of the form \eqref{eq:osfamily}. Integrating the `running cost' term $\int_0^{\t}e^{-\alpha t}\lambda X^2_t dt$ by parts we obtain 
\begin{eqnarray}
\label{Wstar2}
V_0(x) &:=& \inf_{\t \geq 0} \EE\big[ e^{-\alpha\tau}h_l(X_{\tau})\big]
= \tilde V_0(x) - \flas - \fla x^2, \label{eq:protonofuel} \\
V(x;c) &:=&  \inf_{\t \geq 0} \EE\big[ e^{-\alpha\tau}h(X_{\tau};c)\big] 
 = \tilde V(x;c)- \flas - \fla x^2, \label{def-V}
\end{eqnarray}
where the {\em obstacles} $h_l$ and $h$ respectively are given by
\begin{eqnarray}
h(x;c)&:=&h_l(x) \wedge h_r(x;c), \label{eq:hdef}\\
h_l(x)&:=& \lb \d - \fla \rb x^2 -\flas, \label{eq:hldef}\\
h_r(x;c)&:=& \tilde V_0(x-c) -\flas - \fla x^2 + c.
\label{def-G0}
\end{eqnarray}
We solve this family of problems using the characterisation via concavity of excessive functions, as follows. Define $\phi_{\a}(x):=e^{-\sqrt{2\a}x}$ and $\psi_{\a}(x):=e^{\sqrt{2\a}x}$ to be the decreasing and increasing solutions respectively of the characteristic equation $(\cL - \a)u =0$, where $\cL:= \frac{1}{2}\frac{d^2}{dx^2}$ is the infinitesimal generator of Brownian motion. As in \cite{Dayanik2003}, eq.\ (4.6), we set
\begin{equation}
\label{def-F}
\Psi(x):=\frac{\psi_{\a}(x)}{\phi_{\a}(x)} = e^{2\sqrt{2\a}x}, \qquad x \in \RR.
\end{equation}
With an obvious terminology we will refer to the point $x$ as being in the {\em natural scale} and the point $y=\Psi(x)$ as being in the {\em transformed scale}.
Further the {\em transformed obstacle} $H$ is defined by:
\begin{align}\label{def-H}
H(y;c):=
\left\{
\begin{array}{ll}
\frac{h(\Psi^{-1}(y);c)}{\phi_\a(\Psi^{-1}(y))}, & y>0,\\[+4pt]
0, & y=0.
\end{array}
\right.
\end{align}
We refer to \eqref{def-H} as the {\em usual transformation} below and with an obvious notation we also define the transformed obstacle $y \mapsto H_r(y;c)$ by replacing $h$ with $h_r$ in \eqref{def-H}, and so on. To establish the geometric solution to \eqref{def-V} used below we now restate results from Proposition 5.12, Remark 5.13 and Section 6 of \cite{Dayanik2003} as follows:
\begin{Proposition}\label{prop:DayKar} 
\begin{enumerate}
\item Fix $c > 0$ and let $W(\,\cdot\,;c)$ be the greatest non-positive convex minorant of $H(\,\cdot\,;c)$, then $V(x;c)=\phi_\alpha(x)W(\Psi(x);c)$ for all $x\in\RR$. Moreover the optimal stopping region is $\cS_c=\Psi^{-1}(\cS^W_c)$, where $\cS^W_c:=\{y>0:W(y;c)=H(y;c)\}$. 
\item Similarly we obtain $V_0(x)=\phi_\alpha(x)W_0(\Psi(x))$ for all $x\in\RR$, where $W_0$ is the greatest non-positive convex minorant of $H_l$. 
\end{enumerate}
\end{Proposition}

\begin{Remark} Below we use the term `minorant' to mean `greatest non-positive convex minorant'. 
\label{rem:minorantword}
\end{Remark}

Before applying this method of solution 
in Sections \ref{sec:nofuel} and \ref{sec:cpos} below we recall the following useful result from Section 6 of \cite{Dayanik2003}:

\begin{Lemma} \label{lem:geomlemma}
Defining the sign of $0$ to be $0$ we have
\begin{enumerate}
\item $H_{y}(y;c)$ has the same sign as $\left(\frac{h}{\phi}\right)_x(\Psi^{-1}(y);c)$,
\item $H_{yy}(y;c)$ has the same sign as $(\cA - \alpha) h(\Psi^{-1}(y);c)$.
\end{enumerate}
\end{Lemma}

\section{Two regimes without fuel}
\label{sec:nofuel}

When the initial fuel level is $c=0$, the connection to optimal stopping is obvious and the stopping problem is standard. However we confirm the solution here using Proposition \ref{prop:DayKar} since its geometry 
is used in Section \ref{sec:cpos} to study the more involved family of problems given by \eqref{def-V}. The optimal stopping rule in this section will have the following two regimes:
\begin{enumerate}
\item[i)] continue, if $X_t>f_0$,
\item[ii)] stop immediately, if $X_t \in [0, f_0]$.
\end{enumerate}

\begin{Lemma}[cf. Figure \ref{fig:1}]\label{lem:Hlprops} For any $\la \in (0,\a \d)$ 
the function $y \mapsto H_l(y)$ is:
\begin{enumerate}
\item continuous on $[1,\infty)$ with $H_l(1)=-\flas$ and $H_l'(1)=-\frac{\la}{2\a^2}$,
\item 
strictly convex on $[1,\Psi(f_0)]$,
\item \label{lempart:413} strictly decreasing on $[1,\Psi(f_0))$ and strictly increasing on $(\Psi(f_0),\infty)$
\end{enumerate}
\end{Lemma}

\begin{proof}
The first part follows since
\begin{eqnarray}\label{eq:Hl}
H_l(y)=\sqrt{y}\left(-\flas + \frac{\delta - \lambda / \alpha}{8\alpha}(\ln y)^2\right)
, & y > 0.
\end{eqnarray} 
The second part follows from Lemma \ref{lem:geomlemma}, since
\begin{eqnarray}
(\cL - \alpha) h_l(x)&=& \label{hlcc}
\d - (\a \d - \la)x^2,
\end{eqnarray}
which is concave in $x$, positive when $x=0$ and equal to $(\a \d -\la) \left(\frac 1 \a+2\frac{f_0}{\sta}\right)>0$ when $x=f_0$. The third part follows from the same lemma since
\begin{equation}
\left(\frac{h_l}{\phi}\right)_x \frac {e^{-x\sta}} {\sta (\d-\frac{\la}{\a})}=
\begin{array}{lr}
\rho(x)\end{array}
\label{eq:derivs}
\end{equation}
with $\rho(x)$ as in \eqref{eq:defrho}, which is negative when $x \in [0,f_0)$ and positive when $x>f_0$.
\hfill$\Box$ 
\end{proof}

\begin{figure}[htbp]
\begin{center}
\includegraphics[scale=1.1]{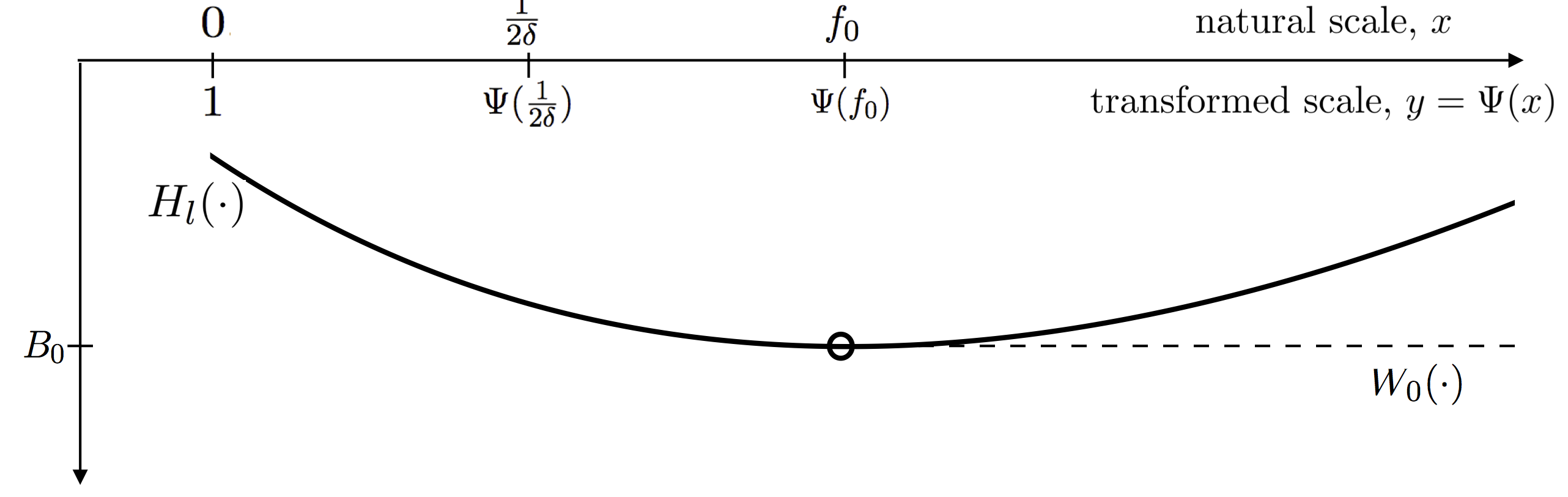}
\caption{A geometric representation of the optimal stopping problem $V_0$ of \eqref{eq:protonofuel}, showing the transformed obstacle $H_l$ on $[1,\infty)$ (solid curve) and its greatest nonpositive convex minorant $W_0$ (dashed curve). The two curves coincide for $y \in [1,\Psi(f_0)]$ and the tangency point is shown (circular marker). For convenience the natural scale $x=\Psi^{-1}(y)$ is also given, and the stopping region in the natural scale is $[0,f_0]$.}
\label{fig:1}
\end{center}
\end{figure}

\vspace{1mm}
Remark \ref{rem:halfspace} suggests that it is sufficient to study the control problem only for $x \geq 0$. The following proof shows that it is also sufficient to study optimal stopping problem \eqref{eq:defq0} only for $x \geq 0$. Further, by restricting the value of $c$ (to $c<c_0$, see Proposition \ref{existencegeometric}
) the same rationale will apply in Section \ref{sec:cpos}.

\begin{Corollary}\label{cor:czerosol}
Let $\tilde W: [0,\infty) \to \R$ be the greatest nonpositive convex minorant of $H_l:[0,\infty) \to \R$. The restriction of $\tilde W$ to the domain $[1,\infty)$ is
\begin{equation}
\tilde W|_{[1,\infty)}(y) = W_0(y):=
\left\{\label{eq:Wzero}
\begin{array}{ll}
H_l(y), & y \in [1,\Psi(f_0)], \\[+1mm]
B_0, 
& y > \Psi(f_0),
\end{array}
\right.
\end{equation}
where $B_0=- \frac{2f_0}{\a \sta}(\a \d - \la)e^{f_0\sta} < 0$.
\end{Corollary}

\begin{proof}
We first argue that the tangent to $H_l(\cdot)$ when $y=1$ lies strictly below $H_l$ for all $y \in [0,1)$. From \eqref{def-H} and \eqref{eq:Hl} we have that $H_l$ is continuous on $[0,1]$ and $H(0)=0$. Further from \eqref{hlcc}, as $y$ increases from $0$ to $1$ (equivalently as $x$ increases from $-\infty$ to $0$), $H_l$ is initially strictly concave and then strictly convex. Combining this with the properties $H_l(1)=-\flas$ and $H_l'(1)=-\frac{\la}{2\a^2}$ establishes this claim. 

The restriction $\tilde W|_{[1,\infty)}$ is therefore the greatest nonpositive convex minorant of the restriction $H_l|_{[1,\infty)}$. This is 
illustrated in Figure \ref{fig:1} and is given by the function $W_0:[1,\infty) \to \R$ defined  in \eqref{eq:Wzero}. Finally, the value of $B_0$ follows from the definition of $H_l$ (it is convenient to use \eqref{def-H}) and \eqref{eq:defrho}.
\end{proof}
\hfill$\Box$

\begin{Remark}\label{rem:solveos}

The function $W_0(\cdot)$ is continuously differentiable. It follows from Proposition \ref{prop:DayKar} and \eqref{eq:protonofuel} that $\tilde V_0$ is also continuously differentiable and
\begin{eqnarray}
\tilde V_0(x) &=& 
\left\{
\begin{array}{ll}
\d x^2, & 0 \leq x \leq f_0, \\[+1mm]
 \flas + \fla x^2 + B_0 e^{-x\sta}, & x > f_0.
\end{array}
\right.
\label{eq:qzero}
\end{eqnarray}
\end{Remark}

\section{Inclusion of fuel}
\label{sec:cpos}
The two regimes in the above optimal stopping rule will now play a key role in the control problem. In particular the new repelling boundary constructed below 
propels the state into {\em either} the `stop' regime {\em or} the `continue' regime from Section \ref{sec:nofuel}. This contrasts with the previously solved range $\la \in (0,\la^*]$ in which the repelling boundary never propels the state directly into the `stop' regime (see Figure \ref{fig:boundaries}); it also contrasts with the previously solved range $\la \geq \a\d$, in which repulsion always propels the state into the `stop' regime (see Section 5 of \cite{KOWZ00} for further details).

Fixing $\a>0$, $\d>0$ and recalling that the open range is $\la \in (\la^*,\a\d)$ it may be checked  
(see \eqref{eq:fodef} and the preceding discussion) that the equation 
$f_0(\la)=\atl$ 
holds at a unique point $\la^\dagger \in (\la^*,\a\d)$. It will become clear below that our optimal stopping heuristic does not apply in the case  $\la \in (\la^*,\la^\dagger)$, and this parameter range remains an interesting open problem.
Below we therefore address the case $\la \in (\la^\dagger, \a \d)$, or equivalently $f_0 > \atl$, via the
family \eqref{def-V} as $c$ varies. Since $\la \in (\la^\dagger,\a\d)$ we have the inequality
\begin{eqnarray}
  \frac 1 {2\d} < \atl < f_0.& \label{eq:longineq}
\end{eqnarray}
We first evaluate the obstacle $h(\cdot;c)$. From \eqref{def-G0} the function $h_r(\cdot;c)$ is continuously differentiable and from \eqref{eq:qzero} it divides at $x=f_0+c$ into two parts corresponding to the two regimes of Section \ref{sec:nofuel}:
 \begin{eqnarray}
h_r(x;c) = \left\{
\begin{array}{ll}
h_{r1}(x;c):=  \d c(c-2x) + (\d - \la / \a) x^2 + c -\flas, & x \in [0,f_0+c],\\
h_{r2}(x;c):= \fla c(c-2x) +  c+ B_0 e^{-(x-c) \sta}, & x \geq f_0+c.
\end{array}
\right.
\label{def-G1}
\end{eqnarray}
\begin{figure}[htbp]
\begin{center}
\includegraphics[scale=1.07]{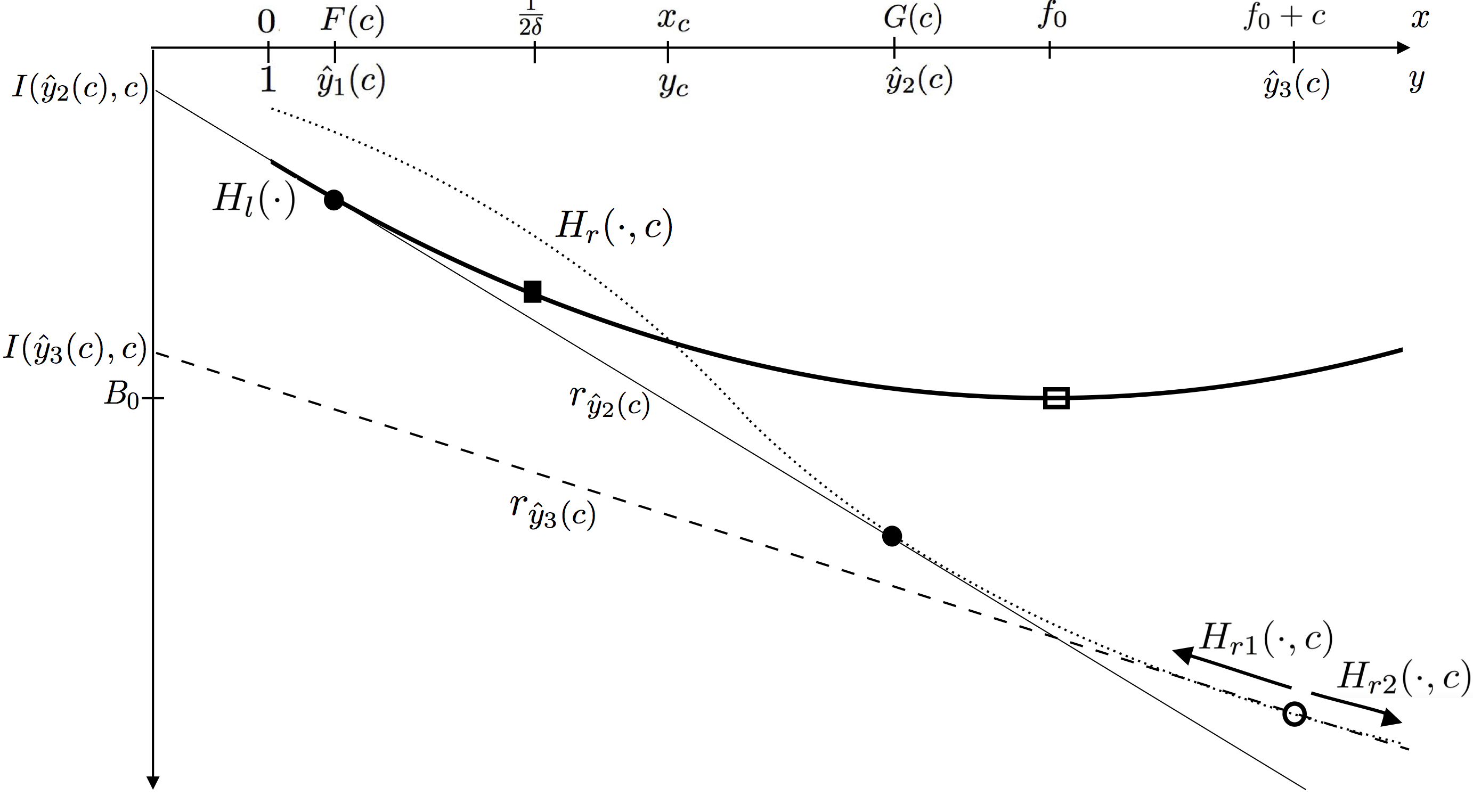}
\caption{A geometric representation of the optimal stopping problem $V(x;c)$ of \eqref{def-V} for $c>0$ fixed and sufficiently small, showing the transformed obstacles $y \mapsto H_l(y)$ (solid curve) and $y \mapsto H_r(y;c)$ (dotted curve). Their common tangent $r_{\hyt(c)}$ (solid line) and the tangent $r_{\hye(c)}$ (dashed line) are also shown, together with their intercepts at the vertical axis (see Definition \ref{def:tangents}). 
As $c \to 0$, $H_r(\cdot;c)$ converges pointwise to the minorant $W_0(\cdot)$ shown in Figure \ref{fig:1}  (see Lemma \ref{lem:minorant}) and $H(\cdot;c):=H_l(\cdot) \wedge H_r(\cdot;c)$ also converges pointwise to $W_0(\cdot)$. Since the filled circular markers converge to the filled square marker (Lemma \ref{cor:1/2d0}), the minorant $W(\cdot;c)$ of $H(\cdot;c)$ also converges pointwise to $W_0(\cdot)$ and the stopping region $[0,F(c)] \cup [G(c),\infty)$ of $V(\cdot;c)$  converges to $\R$ as $c \to 0$. However the stopping region of $V_0(\cdot)$ is $[0,f_0]$ (the unfilled circular marker converges to the unfilled square marker as $c \to 0$). 
}
\label{fig:Fig1synth}
\end{center}
\end{figure}
In particular we have:
 \label{sec:sol1} 
 \begin{Lemma}
 \label{lem:littleh}
Fixing $c>0$ and defining $x_c=\frac 1 {2\d}+\frac c 2$, problem \eqref{def-V} has obstacle:
\begin{eqnarray}
h(x;c) = \left\{
\begin{array}{ll}
h_l(x) & \text{ for } x \leq x_c, \\
h_{r1}(x;c) & \text{ for } x \in [x_c, f_0 + c], \\
h_{r2}(x;c) & \text{ for } x \geq f_0 + c. \\
\end{array}
\right. \label{eq:hbigf0}
\end{eqnarray}
\end{Lemma}

\begin{proof}
To determine the minimum $h_l(x) \wedge h_r(x;c)$ in \eqref{eq:hdef}, it is easily checked from \eqref{def-G1} that the continuously differentiable function $x \mapsto h_r(x;c) - h_l(x)$ is concave. Also for $ x < f_0+c$ we have
\begin{equation}
h_r(x;c) - h_l(x) = h_{r1}(x;c) - h_l(x) 
= c(1+\d(c-2x)), \label{eq:hrhl}
\end{equation}
which is strictly decreasing in $x$ and has the root 
$x_c:=\frac 1 {2\d}+\frac c 2 
< f_0+c.$
\end{proof}
\hfill $\Box$

Defining 
\begin{eqnarray*}
y_c & = & \Psi(x_c), \\
\hye(c)&=&\Psi(f_0+c),
\end{eqnarray*}
the next lemma establishes the relevant geometry of the transformed obstacle $H_r(\cdot;c)$, which is illustrated in Figure \ref{fig:Fig1synth}.
\begin{Lemma}\label{lem:Hconvexity}
The transformed obstacle $y \mapsto H_r(y;c)$ is continuously differentiable and there exists $y_v(c) \in [y_c, \hye(c))$ such that $y \mapsto H_r(y;c)$ is strictly convex on $(y_v(c),\infty)$ and strictly concave on $(y_c,y_v(c))$ (if $y_v(c)>y_c$).  
Further we have
\begin{eqnarray}
\lim_{y \to \infty}H_r(y;c) &=& -\infty, \label{eq:Hneglim} \\
\lim_{y \to \infty}\pd {H_r} y (y;c) &=& 0. \label{eq:hr2lims}
\end{eqnarray}
\end{Lemma}
\begin{proof}
Fixing $c>0$, the first claim follows from the smoothness of the usual transformation. We have
\begin{eqnarray}
(\cL - \alpha) h_{r1}(x;c) &=&
\d - (\a \d - \la)x^2 +  c \a ( \d (2 x -c ) - 1), \label{hr1cc} \\
(\cL - \alpha) h_{r2}(x;c) &=& c(\la (2 x -c ) - \a), \label{hr2cc}
\end{eqnarray}
and at $x=f_0+c$ it is straightforward (using \eqref{eq:defrho}) that both \eqref{hr1cc} and \eqref{hr2cc} are strictly positive.

Since the expression \eqref{hr1cc} is concave, negative as $x \to -\infty$ 
and positive when $x=f_0+c$, it has a unique root $x_v(c) \in (-\infty,f_0+c)$.
Together with \eqref{hr2cc} 
this implies that $(\cL - \alpha) h_r$ changes its sign exactly once on $\R$, namely when $x=x_v(c)$. Setting $y_v(c)=y_c \vee \Psi(x_v(c))$, the required result then follows from Lemma \ref{lem:geomlemma}.
\end{proof}
\hfill$\Box$

\begin{Remark}\label{rem:heurbreak}
It is clear from \eqref{hr2cc} that if $f_0<\atl$ then for sufficently small $c>0$, the transformed obstacle $H_{r2}(\cdot;c)$ would be strictly concave on $[f_0+c, y_z]$ for some $y_z>f_0+c$. It is straightforward to show from the obstacle geometry that this would create a disconnected continuation region in problem \eqref{def-V} for sufficiently small $c$. 
The assumption $f_0 \geq \atl$ (cf. \eqref{eq:longineq}) is therefore necessary for our optimal stopping heuristic to apply directly when $f_0>\fotd$. Interestingly, this concave region would not cause a similar problem in the bottom panel of Figure \ref{fig:boundaries} since it lies {\em inside} the continuation region when $\la \in (\la_*,\la^*]$ and $c \in [0,\bar c]$ (see also Remark \ref{rem:noapply}).
\end{Remark}

We will see below (in Proposition \ref{existencegeometric}) that in contrast to Section \ref{sec:nofuel}, for each $c>0$ the minorant touches the obstacle twice, once on $H_l$ and once on $H_r$. For small $c$ the right-hand tangency point is with $H_{r1}$ (rather than with $H_{r2}$), a situation illustrated in Figure \ref{fig:Fig1synth}. It is this distinction which leads to new equations for the moving boundaries of the control problem, and below we will restrict attention to this new case in which tangency occurs with $H_{r1}$.

In order to construct the minorants $W(\cdot;c)$ to the transformed obstacles $H(\cdot;c)$ we now examine the geometry of $H_r(\cdot;c)$ as $c$ decreases to 0.

\begin{Definition}\label{def:tangents}
 For $y \neq y_c$ let $r_y(\,\cdot\, , c)$ be the straight line tangent to $H(\cdot;c)$ at $y$, and let $r_{y_c}(\,\cdot\, , c)$ be the straight line tangent to $H_r(\cdot;c)$ at $y_c$:
\begin{align}
r_y(z;c)&=\pd H y (y;c)(z-y)+H(y;c), \quad y \neq y_c,\nonumber
\\[+5pt]
&=\pd {H_r} y (y;c)(z-y)+H(y;c), \quad y = y_c.\nonumber
\end{align}
Let $I(y;c)=r_y(0;c)$ be its intercept at the vertical axis, and let $P_r(y;c)$ be the following signed distance between $r_y(\,\cdot\, , c)$ and $H_l$:
\begin{align}\
P_r(y;c)&=\sup_{z \in [1, y_c]}a(z,y;c),\:\: \text{ where }\:\: a(z,y;c):=r_y(z;c)-H_l(z).\label{def-P_r}
\end{align} 
\end{Definition}

\begin{Lemma}
\label{lem:minorant}
Let $1 < y_u<y_w \leq \Psi(f_0)$. Then for all $c>0$ sufficiently small, the tangent $r_{y_u}(\cdot;0)$ to $H_l(\cdot)$ at $y_u$ lies strictly below $H_r(\cdot;c)$ on $[y_w,\infty)$.
\end{Lemma}
\begin{proof} 
We have from \eqref{def-G0} and \eqref{eq:protonofuel} that $h_r(x;0) = V_0(x)$ for $x \geq 0$. From \eqref{def-G1} we have
\begin{equation*}
h_{r}(x;c)-h_{r}(x;0) = h_{r}(x;c)-V_0(x) = c \cdot O(x) \quad \text{ as } x \to \infty, 
\end{equation*}
so applying the usual transformation and Proposition \ref{prop:DayKar} we have
\begin{equation}\label{eq:Hrconv}
|H_{r}(y;c)-W_0(y)|=c \cdot O(\sqrt y) \text{ as } y \to \infty.
\end{equation}
From Figure \ref{fig:1} the tangent $r_{y_u}(\cdot;0)$ 
lies strictly below $W_0(y) = H_r(y,0)$ for $y \geq y_w$ and has strictly negative gradient. The result follows from the estimate \eqref{eq:Hrconv}.
\hfill$\Box$
\end{proof}

\vspace{1mm}
We now present the main result used in this paper, which characterises the new moving boundaries $F$ and $G$ illustrated in the top panel of Figure \ref{fig:boundaries} and confirms that the principle of smooth fit holds across them.

\begin{Proposition}
\label{existencegeometric}
There exists $c_1>0$ such that for each fixed $c \in (0,c_1)$ the following three statements hold:

1. There exists a unique couple $(\hat{y}_1(c),\hat{y}_2(c))$ with $1 < \hat{y}_1(c) < y_c \leq y_v(c) < \hat{y}_2(c)$ solving the system
\begin{equation}
\label{eq:system-y1y2}
\left\{
\begin{array}{lr}
\pd {H_l}{y}(y_1)= \pd{H_r} y (y_2;c), \\[+4pt]
H_l(y_1) - \pd {H_l} y (y_1)y_1 = H_r (y_2;c) -\pd {H_r} y (y_2;c)y_2.
\end{array}
\right.
\end{equation}
We also have the bound 
\begin{equation}\label{eq:hyobound}
\hyo(c)<\Psi(f_0).
\end{equation}

2. The function $y \mapsto W(y;c)$ given by
\begin{equation}
W(y;c) := 
\left\{
\begin{array}{ll}
H_l(y), & 1 \leq y \leq \hyo(c),\\[+3pt]
A(c)y + B(c),
& \hyo(c) < y < \hyt(c), \\[+3pt]
 H_r(y), & y \geq \hyt(c),\\
\end{array}
\right. \label{eq:Wdef}
\end{equation}
where $A(c)=\pd {H_l}{y}(\hyo(c))=\pd {H_r}{y}(\hyt(c);c)$ and $B(c)=I(\hyo(c);c)=I(\hyt(c);c)$ (see Definition \ref{def:tangents}), is the greatest nonpositive convex minorant of $H(\cdot;c)$ on $[1,\infty)$ and lies in $C^1[1,\infty)$. The straight line 
\begin{equation}\label{eq:line}
y \mapsto A(c)y + B(c), \quad y>1,
\end{equation}
also lies below $H(\cdot;c)$.

3. Setting
\begin{eqnarray}
\begin{array}{ll}
F(c)&=\Psi^{-1}(\hyo(c)),\\ G(c)&= \Psi^{-1}(\hyt(c)),
\end{array} 
\label{eq:FGdef}
\end{eqnarray}
we have
\begin{equation}
\tilde V(x;c) := 
\left\{
\begin{array}{ll}
\d x^2, & 0 \leq x \leq F(c),\\[+3pt]
\fla x^2 + \flas + A(c) e^{x\sta} + B(c) e^{-x\sta}, & F(c) < x < G(c), \\[+3pt]
 \tilde V_0(x-c,0) + c, & x \geq G(c),\\
\end{array}
\right. \label{eq:Qdef}
\end{equation}
with $\tilde V_0$ defined as in \eqref{eq:qzero}. The principle of smooth fit holds across the free boundary points $F(c)$ and $G(c)$.
\end{Proposition}

\begin{Remark} Figure \ref{fig:Fig1synth} contains sufficient information to suggest the proof of parts 1 and 2. Also we will refer to the line \eqref{eq:line}, which may also be denoted $r_{\hyo(c)}(\cdot;c)$ or $r_{\hyt(c)}(\cdot;c)$, as the {\em common tangent}.
\end{Remark}

\begin{proof} 
1. It is clear from Figure \ref{fig:Fig1synth} (or alternatively by using  \eqref{eq:hrhl}) that 
the function $y \mapsto P_r(y;c)$ of Definition \ref{def:tangents} is strictly positive at $y=y_c$. By the concavity of $H_r$ on $(y_c,y_v(c))$, $P_r(\cdot;c)$ is increasing on $(y_c,y_v(c))$ and by convexity it is decreasing on $(y_v(c),\infty)$. 
Evaluating the transformed obstacle $H_{r2}(y;c)$ using \eqref{def-G1} we obtain 
\begin{equation}
H_{r2}(y;c)= \sqrt y \left( \fla c\left( c-\frac {\ln y}{\sta} \right) +  c+ B_0 e^{c \sta}y^{-1/2}\right),
\end{equation}
so that the intercept $I(y;c)$ satisfies
\begin{equation}\label{eq:Ilim}
I(y;c)=-\frac{\sqrt y \la c}{(2\a)^{3/2}}\ln y+ O(\sqrt y) \to -\infty \text{ as } y \to \infty.
\end{equation} 
From Figure \ref{fig:Fig1synth} (see also Lemma \ref{lem:Hconvexity}) it is now clear that 
$\lim_{y \to \infty}P_r(y;c)=-\infty$. We conclude from the continuity of the function $y \mapsto P_r(y;c)$ that there exists a unique point $\hat{y}_2(c) \in (y_v(c),\infty)$ satisfying $P_r(\hat{y}_2(c);c)=0$. By strict convexity we also have  
\begin{equation} \label{eq:Hcvlim}
H(y;c)-r_{\hyt(c)}(y;c) > 0 \quad \text{ for all } y>\hyt(c).
\end{equation}
By the convexity of $H_l$, the supremum in \eqref{def-P_r} must be uniquely achieved at some point $z=:\hat{y}_1(c) \in [1,y_c)$. The couple $(\hyo(c),\hyt(c))$ then satisfies equations \eqref{eq:system-y1y2} if $\hyo(c)$ lies in the interior of this interval, ie. if $\hyo(c) > 1$. We therefore seek $c_1>0$ such that
\begin{equation}\label{eq:hyoneq1}
\hyo(c) > 1 \text{ for all } c \in (0,c_1).
\end{equation} 
Fix $c>0$ and suppose, hoping for a contradiction, that $\hyo(c)=1$. By construction the common tangent then lies {\em strictly} below 
$H_l(z)$ for $z \in (1,y_c]$, 
so in particular it must lie below (or be equal to) the tangent $r_{1}(y;c)$ 
for all $y \geq 1$. 
However from 
Lemma \ref{lem:minorant}, for sufficiently small $c$ 
we have $r_1(y;c)< H_r(y;c)$ for all $y \in [\Psi(\fotd),\infty)$. Therefore the common tangent lies strictly below $H_r(\cdot;c)$ at the supposed tangency point $y=\hyt(c)>\Psi(\fotd)$, a
contradiction.
We conclude that %
$\hyo(c)>1$ for sufficiently small $c$, and so 
there exists $c_1>0$ such that \eqref{eq:hyoneq1} 
holds. 

Recall also that $H_l(y)$ is increasing on $[\Psi(f_0),\infty)$ from Lemma \ref{lem:Hlprops}. If the tangency point $\hyo(c)$ lay in this interval then as $y \to \infty$ the common tangent would increase to infinity while $H$ decreased without bound (from \eqref{eq:Hneglim}), 
contradicting \eqref{eq:Hcvlim}, and so we conclude that
 $\hyo(c)<\Psi(f_0)$.
 
2. Now fix $c \in (0,c_1)$ and construct $\hyo(c)$ and $\hyt(c)$ as above. It is clear from Figure \ref{fig:Fig1synth} that in order to establish that the function $W(\cdot;c)$ of \eqref{eq:Wdef} is the greatest nonpositive convex minorant of $H(\cdot;c)$ on $[1,\infty)$ it remains only to check that the tangent $r_{\hyo(c)}(\cdot;c)$ lies below $H(y;c)$ for $y \in (\hyo(c),\hyt(c))$. Suppose for a contradiction that the common tangent coincides with $H(\cdot;c)$
when $y=y_a \in (\hyo(c),\hyt(c))$. It is obvious from the geometry of $H_l$ that 
then $y_a>y_c$; also $H_r$ must be concave at 
$y_a$. This forces $H(\cdot;c)$ to lie strictly below the common tangent when $y=y_c$, a contradiction.
Finally the `double tangency' equations \eqref{eq:system-y1y2} confirm that $W(\cdot;c)$ lies in $C^1(1,\infty)$. 

3. This part now follows from Proposition \ref{prop:DayKar}.
\end{proof} \hfill $\Box$


\section{Free boundary analysis} \label{sec:sol1fam}

In this section we establish the existence of $c_0>0$ such that the moving boundaries $F(c)$ and $G(c)$ may be extended to continuous monotone functions on $[0,c_0)$, which are differentiable on $(0,c_0)$ with $G'(c)>1$ (recall Figure \ref{fig:boundaries}), but $F(0) \neq f_0$. 

\begin{Proposition}
\label{cor:1/2d0}
The functions $\hyo$, $\hyt$ defined in Proposition \ref{existencegeometric} satisfy $\lim_{c \to 0}\hyo(c)=\lim_{c \to 0}\hyt(c)=\Psi(\frac 1 {2\d})$.
\end{Proposition}
\begin{proof} We employ Lemma \ref{lem:minorant} as in the proof of part 1 of Proposition \ref{existencegeometric}. We have
\[
\limsup_{c \to 0} \hyo(c) \leq \Psi \lb \frac 1 {2\d} \rb \leq \liminf_{c \to 0} \hyt(c).
\] 
Suppose that $\liminf_{c \to 0} \hyo(c) = y^* < \Psi \lb \frac 1 {2\d} \rb$. Taking $y_u=y^*$ and $y_w=\Psi(\fotd)$ in Lemma \ref{lem:minorant} again gives a contradiction for sufficiently small $c$. 

Similarly suppose that $\limsup_{c \to 0} \hyt(c) = y^\circ > \Psi \lb \frac 1 {2\d} \rb$. Then taking $y_u=\Psi(\fotd)$ and $y_w$ with $\Psi(\fotd) < y_w < \min\{y^\circ,\Psi(f_0)\}$ gives the required contradiction for sufficiently small $c$. 
 \hfill $\Box$
\end{proof}

\vspace{1mm}
We now analyse the moving boundaries $F$ and $G$. Where possible we will reuse identities and notation from \cite{KOWZ00} (in particular the notation $h_1$ to $h_4$, $\cH_3$, $\cH_4$, $q$, $L$ and $U$) so that the analytical contribution of the present paper is clear.

\subsection{Smooth fit equations}

It is straightforward to check from \eqref{eq:Hl} that the straight line tangent to $H_l(\cdot)$ at $y=\Psi(x)$ has gradient
\begin{eqnarray}
\label{eq:hidefs1}
h_1(x) &:=& \frac{\a \d -\la}{2\a}e^{-x\sta}\rho(x), 
\end{eqnarray}
with $\rho$ defined as in \eqref{eq:defrho}. 

\begin{Remark}
\label{rem:H3c} 
It follows that $\pd {h_1}{x}(x)$ 
 is strictly positive when $y=\Psi(x)$ lies in a strictly convex region of $H_l(\cdot)$. From Lemma \ref{lem:Hlprops} we may therefore define $h_1^{-1}:[-\frac{\la}{2\a^2},0] \to [0,f_0]$. Recalling the definition of $f_0$ (see \eqref{eq:defrho}), $h_1^{-1}$ is differentiable on $(-\frac{\la}{2\a^2},0)$. 
\end{Remark}
This tangent line intercepts the vertical axis of Figure \ref{fig:Fig1synth} at the value 
\begin{eqnarray}\label{eq:hidefs2}
h_2(x) &:=& \frac{\a \d -\la}{2\a}e^{x\sta}\left(\rho(x) - \frac{4x}{\sta} \right).
\end{eqnarray}
The usual transformation gives $H_{r1}(y;c)=e^{x\sta} h_{r1}(x;c)$, where $y=\Psi(x)$.
It follows that the line tangent to $H_{r1}(\cdot;c)$ at $y=\Psi(x)$ has gradient $\tilde \cH_3(x,c)$, where
\begin{equation}
\tilde \cH_3(x,c):=\pd {H_{r1}}y (y;c)=\frac{e^{-x\sta}}{2\sta}\lb \pd {h_{r1}}x (x;c) + \sta {h_{r1}}(x;c)\rb, \label{eq:th3}
\end{equation}
 while its intercept at the vertical axis is
\begin{equation}\label{eq:th4}
 \tilde \cH_4(x,c):=\lb H_{r1}- y \pd {H_{r1}}y \rb (y;c)=\frac{e^{x\sta}}{2\sta}\lb-\pd {h_{r1}}x (x;c) + \sta {h_{r1}}(x;c)\rb.
\end{equation}
Writing 
\begin{eqnarray}
h_3(x)&=&\fla \lb \atl - \lb x+\frac 1 \sta \rb \rb e^{-x\sta}, \nonumber\\
h_4(x)&=&\fla \lb x - \lb \atl +\frac 1 \sta \rb \rb e^{x\sta},
\end{eqnarray}
differentiation then yields the following identities which will be useful below:
\begin{eqnarray}
\lb \pd {\tilde \cH_3} x + \pd {\tilde \cH_3} c \rb (x,c)&=& h_3(x) - \sta \tilde \cH_3 (x,c),\label{eq:hall1}\\
\lb \pd {\tilde \cH_4} x + \pd {\tilde \cH_4} c \rb (x,c)&=&  \sta \tilde \cH_4 (x,c)-h_4(x).
\label{eq:hall2}\end{eqnarray}

We will henceforth restrict attention to the situation illustrated in Figure \ref{fig:Fig1synth}, in which $G(c) <  f_0+c$. In particular this means that the common tangent has gradient $\tilde \cH_3(G(c),c)$. 
This is not a technical restriction necessary for our heuristic to apply, but is imposed so that the moving boundary equations of this paper differ from those identified in the previously solved cases. 
\begin{Remark}\label{rem:noapply}
The repelling boundary $\{G(c):c \in (0,\bar c]\}$ illustrated in the bottom panel of Figure \ref{fig:boundaries} was found for the parameter range $\la \in (\la_*,\la^*]$ using the ansatz that $G(c)>\atl + c$ and, since $f_0 \leq \fotd < \atl$ when $\la \leq \la^*$, it then followed that $G(c)>f_0 + c$. In the latter case the common tangent had gradient $\cH_3(G(c),c)$ (defined in equation (10.9) of \cite{KOWZ00}), which differs from $\tilde \cH_3(G(c),c)$ defined above.
\end{Remark}

Define

\begin{equation}\label{eq:Ldot}
L(x,c):=
\tilde \cH_4(x,c)
-h_2\circ h_1^{-1}\left(\tilde \cH_3(x,c)\right), \end{equation} 
where $x$, $c$ 
must be chosen from a suitable domain. 

\begin{Lemma}\label{rem:Linterp}
$L(f_0+c,c)$ is well defined for all $c \in (0,c_1)$ and the following are equivalent:
\begin{enumerate}
\item[a)] $L(f_0+c,c)<0$,
\item[b)]  the line $r_{\hat y_3(c)}(\cdot;c)$ (see Figure \ref{fig:Fig1synth}) 
lies strictly below $H_l(\cdot)$,
\item[c)] $\hyt(c)<\hat y_3(c)$.
\end{enumerate}
\end{Lemma}

\begin{proof} Since $\hat y_3(c)$ is the boundary between $H_{r1}(\cdot;c)$ and $H_{r2}(\cdot;c)$ (see Figure \ref{fig:Fig1synth}), and since $H_r(\cdot;c)$ is continuously differentiable (Lemma \ref{lem:Hconvexity}), we have $\tilde \cH_3(f_0+c,c) = \cH_3(f_0+c,c)$.
We recall from page 46 of \cite{KOWZ00} that $$\cH_3(x,c) \in \left(-\frac{\la}{2\a^2},0\right) \text{ for all } x > \atl + \frac c 2,$$ and the latter inequality holds when $x=f_0+c$ by assumption \eqref{eq:longineq}. The first claim follows. Figure \ref{fig:Fig1synth} easily establishes the equivalence between a) and b), as follows. Recalling the definitions \eqref{eq:hidefs1}-\eqref{eq:hidefs2}, 
the second term on the right hand side of \eqref{eq:Ldot} (putting $x=f_0+c$) is the vertical intercept of the unique line tangent to $H_l(\cdot)$ whose gradient is $\cH_3(f_0+c,c)$, which is precisely 
the gradient of $r_{\hat y_3(c)}(\cdot;c)$. 
Recalling from Lemma \ref{lem:Hconvexity} that $H_{r2}(\cdot;c)$ is strictly convex on $[\hat y_3(c),\infty)$, b) is clearly equivalent to c).
\end{proof}
\hfill $\Box$

\begin{Corollary}
The quantity 
\begin{equation}\label{eq:c2def}
c_2 := c_1 \wedge \inf\{c \in (0,c_1): L(f_0+c,c) \geq 0\}
\end{equation}
is strictly positive.
\end{Corollary}

\begin{proof}
By Proposition \ref{cor:1/2d0}, for $c$ 
sufficiently small we have $\hyt(c)<\hat y_3(c)$. 

\end{proof}
\hfill $\Box$


\subsection{Differentiability and monotonicity}

Let us write 
\begin{equation}\label{eq:Vstar}
V^*(x,c) := \fla x^2 + \flas + A^*(c) e^{x\sta} + B^*(c) e^{-x\sta}
\end{equation}
for arbitrary differentiable functions $A^*$ and $B^*$ (cf. \eqref{eq:Qdef}). If we
\begin{enumerate}
\item[i)]  impose smooth fit between $V^*(\cdot,c)$ and the function $x \mapsto \d x^2$ across a free boundary point $F(c)$ as in \eqref{eq:Qdef}, 
\item[ii)] impose the assumptions $U(G(c),c)=1$, $U_x(G(c),c)=a(c)$, where 
\begin{equation}\label{eq:UUdef}
U(x,c):=(V^*_{x}+V^*_{c})(x,c),
\end{equation}
\end{enumerate}
then it follows from straightforward manipulations that   
\begin{equation}\label{eq:C15}
2 \sta e^{G(c)\sta}q(G(c);F(c)) = (e^{2G(c)\sta}-e^{2F(c)\sta})a(c),
\end{equation}
where 
\begin{eqnarray}\label{eq:qdef}
q(x;z)&=& \sta (h_2(z)-h_1(z)e^{2z\sta}) + h_3(x)e^{2z\sta} - h_4(x) \\
&=& -2z\frac{\a\d-\la} \a e^{z\sta} + \fla e^{x\sta} \left[\left(\atl -x -\frac 1 \sta \right)e^{2(z-x)\sta} \right.\nonumber\\ && \qquad + \left.\left(\atl -x +\frac 1 \sta \nonumber
\right)
\right].
\end{eqnarray}
\begin{Remark} A detailed derivation of \eqref{eq:C15} is given in (8.10)-(8.22) of \cite{KOWZ00} in the case $a(c)=0$. When $a(c) \neq 0$, \eqref{eq:C15} is equation (C.15) of the latter paper. Alternatively the validity of \eqref{eq:C15} in the present setting follows directly from the arbitrariness of $A^*(\cdot)$ and $B^*(\cdot)$ in \eqref{eq:Vstar}, which equal $\cH_3(G(\cdot),\cdot)$ and $\cH_4(G(\cdot),\cdot)$ respectively in the latter paper but equal $\tilde \cH_3(G(\cdot),\cdot)$ and $\tilde \cH_4(G(\cdot),\cdot)$ respectively in the present work.
\end{Remark}
We note here that 
\begin{eqnarray}
q(z;z)&=&e^{z\sta}(1-2\d z), \label{eq:qprop1}\\
q_x(x;z)&=& \la \sqrt{\frac 2 \a}\left( \atl - x \right)e^{x\sta}(1-e^{2(z-x)\sta}), \quad z \leq x < \infty. \label{eq:qprop2}
\end{eqnarray}

\begin{Proposition}\label{cor:1/2d}
Defining $\hyo(0)=\hyt(0)=\Psi(\fotd)$, the functions $\hyo$, $\hyt$ defined in Proposition \ref{existencegeometric} both lie in $C[0,c_2)\cap C^1(0,c_2)$. 
\end{Proposition}
\begin{proof}
It is sufficient to apply the Implicit Function Theorem to system \eqref{eq:system-y1y2} with respect to the independent variable $c$ as in Appendix A of \cite{DeAFeMo14}. Here the corresponding matrix of derivatives  
has determinant $$D(\hyo(c),\hyt(c),c)=(\hyt(c)-\hyo(c))\pd {^2H_l}{y^2}(\hyo(c)) \pd{^2H_r} {y^2} (\hyt(c);c).$$ 
Recalling Proposition \ref{existencegeometric}, this determinant is strictly positive for $c \in (0,c_2)$ since then $\hyo(c) \in (1,\Psi(f_0))$ (by \eqref{eq:hyoneq1}) and $\hyt(c) \in (y_v(c),\hat y_3(c))$ (by Lemma \ref{rem:Linterp}), and the functions $H_l(\cdot)$, $H_r(\cdot;c)$ ($=H_{r1}(\cdot;c)$) are twice differentiable and strictly convex on these respective intervals. The limit established in Proposition \ref{cor:1/2d0} completes the argument.
\end{proof}
 \hfill $\Box$
 
\vspace{1mm}
We will now define $c_0$ such that $G'(c)>1$ on $(0,c_0)$. 
Note first that $L(G(c),c)$ is well defined and identical to 0 for each $c \in (0,c_2)$ because of the double tangency in Figure \ref{fig:Fig1synth}. It also follows by the continuity of $(x,c) \mapsto \tilde \cH_3(x,c)$ that the map 
$(x,c) \mapsto L(x,c)$ is well defined in an open neighbourhood of $(G(c),c)$ and by Remark \ref{rem:H3c} it is differentiable in this neighbourhood. Writing $z=h_1^{-1}(\tilde \cH_3(x,c))$, the following formulae are established in Appendix  \ref{app:hall} (where again $x$, $c$ must be chosen from a suitable domain):
 \begin{eqnarray}\label{eq:lx}
L_x (x,c) &=& \pd {\tilde \cH_3} x (x,c) (e^{2z\sta}-e^{2x\sta}),\\
(L_x+L_c)(G(c),c)&=& q(G(c),F(c)).\label{eq:hall3}
\end{eqnarray}

\begin{Lemma}\label{cor:FGC1}
For $c \in (0,c_2)$ we have $L_x(G(c),c)<0$ and 
\begin{equation}\label{eq:Gderiv}
G'(c) = 1 - \frac{q(G(c);F(c))}{L_x(G(c),c)}.
\end{equation}
\end{Lemma}

\begin{proof}
Fix $c \in (0,c_2)$. By strict convexity $\pd{}{x} \tilde \cH_3(G(c),c)$ is strictly positive 
and it follows from \eqref{eq:lx} that 
$L_x(G(c),c)< 0$. Differentiation of the identity
\begin{equation}\label{eq:Lis0}
L(G(c),c) \equiv 0 \text{ for all } c \in (0,c_2)
\end{equation}
with respect to $c$ and using \eqref{eq:hall3} then gives \eqref{eq:Gderiv}.
\hfill$\Box$
\end{proof}

\vspace{1mm}
The sign of $q(G(c);F(c))$ may be established from \eqref{eq:qprop1} and \eqref{eq:qprop2} using the following lemma. 
\begin{Lemma} \label{lem:Fcsmall}
For sufficiently small $c>0$ we have $F(c)<\frac 1 {2\d}$.
\end{Lemma}
\begin{proof}
Assume that 
$\fotd \leq F(c)$ where $c \in (0,c_2)$.
In the transformed coordinates of Figure \ref{fig:Fig1synth}, by Proposition \ref{existencegeometric} we then have 
\begin{equation} \label{eq:llongineq}
\Psi \lb\fotd\rb\leq \Psi(F(c)) = \hyo(c) < y_c = \Psi\lb\fotd + \frac c 2\rb < \Psi\lb\fotd + c\rb =: y_r.
\end{equation}
It then follows from the convexity of $H_l$ that the common tangent lies above the tangent $ r_{\Psi(\fotd)}(\cdot;c)$ at $y=y_r$, that is
\begin{equation}\label{eq:dom}
r_{\hyo(c)}(y_r;c) \geq r_{\Psi(\fotd)}(y_r;c). 
\end{equation}
The common tangent also lies below $H(\cdot;c)$. Putting $x_r=\frac 1 {2\d} + c$ and using Propositions \ref{prop:DayKar} and \ref{existencegeometric}  and equation \eqref{def-V} this implies that 
\begin{eqnarray*}
V(x_r;c) &\geq& A(c) e^{x_r\sta} + B(c) e^{-x_r\sta} \nonumber\\
&\geq& A\lb\fotd \rb e^{x_r\sta} + B \lb \fotd \rb e^{-x_r\sta},
\end{eqnarray*}
where the second inequality follows from \eqref{eq:dom}. Since $x_r \in (x_c, f_0+c)$ it follows from the first equation in \eqref{def-G1} that $h(x_r;c)=-\flas - \fla x_r^2 + c + \frac 1 {4\d}$. Since by definition $V(x_r;c) \leq h(x_r;c)$, the estimate 
\begin{eqnarray*}
&&V(x_r;c) - h(x_r;c) \\
&& \; \geq A\lb\fotd \rb e^{(\fotd+c)\sta} + B \lb \fotd \rb e^{-(\fotd+c)\sta} + \flas + \fla \lb \fotd+c \rb^2 - c - \frac 1 {4\d}\\ && \; = \frac{\a\d-\la}{4\d^2}c^2 + O(c^3)
\end{eqnarray*}
establishes the required conclusion. \end{proof} \hfill $\Box$

\vspace{1mm}
Finally we define $c_0$ and establish the monotonicity of the moving boundaries.
\begin{Proposition}\label{prop:monotonicity}Defining
\begin{equation}\label{lem:Gderv}
c_0:= c_2 \wedge \inf \{c \in (0,c_2): G'(c) \leq 1\},
\end{equation}
we have $c_0>0$ and the moving boundaries $c \to F(c)$ and $c \to G(c)$ are respectively decreasing and increasing for $c \in (0,c_0)$.
\end{Proposition}

\begin{proof}
We have established that for sufficiently small $c$: 
\[
F(c)<\fotd  \quad  \text{ and } \quad G(c)<\atl
\]
(for the second inequality recall \eqref{eq:longineq}). Taking $z=F(c)$ with $c$ fixed and sufficiently small, from \eqref{eq:qprop1} we have 
$q(F(c),F(c))>0$ and 
 for $x \in (F(c), G(c))$ we have $q_x(x;F(c))>0$ from \eqref{eq:qprop2}. 
Then $q(G(c);F(c))>0$ and $G'(c)>1$ (Lemma \ref{cor:FGC1}). We conclude that $c_0>0$.

The idea used in the proof of Lemma \ref{lem:Fcsmall} provides a convenient way to establish the monotonicity of $F$, as follows. Fixing $\tilde c \in (0,c_0)$, we note that the common tangent and the transformed obstacle coincide when $y=\tilde y:=\hyt(\tilde c)$. We then fix $y=\tilde y$ and vary $c$ locally around $\tilde c$. 

From Proposition \ref{cor:1/2d} we have  $G(\tilde c) > \frac 1 {2\d}+\tilde c$. Differentiation of the first equation in \eqref{def-G1} yields $\pd {h_{r1}} c (G(\tilde c); \tilde c)=1-2\d(G(\tilde c)-\tilde c)<0$, so that also $\pd {H_{r1}} c (\hyt(\tilde c); \tilde c)<0$ by the usual transformation: the value of the transformed obstacle at  $\tilde y$ is locally decreasing in $c$. However the transformed obstacle $H_l$ is convex and does not depend on $c$; suppose for a contradiction that $\hyo'(\tilde c) \geq 0$. Then at $\tilde y$, the value of the common tangent is locally increasing in $c$. This is a contradiction since the common tangent must lie below the transformed obstacle for all $c$ (Proposition \ref{existencegeometric}).
\hfill$\Box$
\end{proof}

\section{Verification}
\label{Sec:verification}

We now verify the optimality of the following policy for initial states $(x,c) \in [0,\infty) \times [0,c_0)$, which is illustrated in the top panel of Figure \ref{fig:boundaries}:
\begin{enumerate}
\item[a)] If the fuel level is $c=0$:
\begin{enumerate}
\item[i)] stop immediately if $0 \leq X_t \leq f_0$, 
\item[ii)] continue if $X_t>f_0$.
\end{enumerate}
\item[b)] If the fuel level is $c \in (0,c_0)$:
\begin{enumerate}
\item[i)] stop immediately if $0 \leq X_t \leq F(c)$, 
\item[ii)] continue without expending fuel if $F(c)<X_t<G(c)$,
\item[iii)] if $X_t \geq G(c)$ then expend all available fuel immediately and  proceed as in part a).
\end{enumerate}
\end{enumerate}

More precisely, for $c>0$ define the stopping times 
\begin{equation}
\begin{array}{ll}
\tau_c=\inf\{t \geq 0: X_t \in [0,F(c)]\}, &
\sigma_c=\inf\{t \geq 0: X_t \geq G(c)\}, \\
\tau_0=\inf\{t \geq \sigma_c: X_t-c \in [0,f_0]\}, & 
\tau_c^*=\min \{  \tau_{c}, \tau_{0} \},\\
\tau^*=\min \{  \tau_c, \sigma_c \},&
\end{array}
\end{equation}
which almost surely are finite with $\tau^* \leq \tau_c^*$.
Define $\xi^c \in \cA(c)$ by $\xi^c_t=-c1_{t \geq \sigma_c}, \; t \geq 0$. Then setting $\tau=\tau_c^*$, $\xi=\xi^c$ in the control problem we obtain
\[
Y_t = \begin{cases}
X_t, & t \in [0,\sigma_c), \\ 
\tilde X_t, & t \geq \sigma_c,
\end{cases}
\]
where $\tilde X_t:=X_t-c$. 
The expected cost of this policy is (cf. \eqref{eq:defV}):
\begin{multline}\label{eq:sm}
\EE\left[ \int_0^{\tau^*} e^{-\alpha t}\lambda X^2_t dt  + e^{-\alpha\tau^*}\delta X^2_{\tau_c}1_{\tau^* = \tau_c}
\right. \\ \left. + e^{-\alpha \tau^*}  1_{\tau^* = \sigma_c} \lb c + \EE\left[ \int_{0}^{\tau_0-\sigma_c} e^{-\alpha s}\lambda \tilde X^2_s ds +  e^{-\alpha(\tau_0-\sigma_c)}\delta \tilde X^2_{\tau_0}  
\right] \rb \right]. \end{multline}
Setting $s = t - \sigma_c$, by the strong Markov property the process $(\tilde X_s)_{s \geq 0}$ is a Brownian motion starting at $X_{\sigma_c}-c$. 
Recalling the definition of $\tau_0$, the `inner' expectation in \eqref{eq:sm} is therefore equal to $V_0(X_{\sigma_c}-c)$ and so \eqref{eq:sm} equals
\begin{multline} \label{eq:markov}
\EE\left[ \int_0^{\tau^*} e^{-\alpha t}\lambda X^2_t dt  + e^{-\alpha\tau^*} \lb \delta X^2_{\tau_c}1_{\tau^* = \tau_c}
 +  \lb c  + V_0(X_{\tau^*}-c) \rb 1_{\tau^* = \sigma_c}
 \rb \right]. \end{multline}
By construction the stopping time $\tau^*$ is optimal in problem \eqref{eq:defqc} (Proposition \ref{prop:DayKar}). By considering separately the events $\{\tau^* = \tau_c\}$ and $\{\tau^* = \sigma_c\}$ and glancing at Figure \ref{fig:Fig1synth}, it is therefore easy to see that the expectation \eqref{eq:markov} is equal to the value function \eqref{eq:defqc}. We conclude by the suboptimality of the pair $(\tau_c^*,\xi^c)$ for the control problem that $\tilde V(x;c) \geq Q(x;c)$.

For the reverse inequality we will apply the following verification theorem:

\begin{theorem}
\label{th:ver}
Consider a positive constant $c_0$ and an evenly symmetric function $\tilde Q: \R \times [0,c_0) \to [0,\infty)$ which is continuous and continuously differentiable, as well as twice continuously differentiable in its first argument with locally bounded second derivatives away from $\{(x,c): x \in \{f(c), g(c), f_0+c\}\}$, for some continuous functions $f$, $g :[0,c_0) \to (0,\infty)$ 
with 
\begin{eqnarray*}
&f, \; g \in C^1((0,c_0)), \; f'(\cdot) \leq 0, g'(\cdot)>0,\\
&f(0) \leq g(0), \text{ and }f(c)<g(c), \forall c \in (0,c_0). 
\end{eqnarray*}
Suppose that $\tilde Q(\cdot,c)$ satisfies the growth condition
\begin{gather}\label{eq:42}
|\tilde Q_x(x,c)| \leq K(c)(1+|x|), \qquad \forall \; x \geq 0, c \in [0,c_0),
\end{gather}
for some continuous and increasing function $K:[0,\infty) \to (0,\infty)$. Suppose also that $\tilde Q$ satisfies the following conditions:
\begin{gather}
\tilde Q(x,c) \leq \d x^2, \quad x\geq 0, c \in [0,c_0); \label{eq:43} \\
\tilde Q_x(x,c) + \tilde Q_c(x,c) \leq 1, \quad x \geq 0, c \in (0,c_0); \label{eq:44}\\
\a \tilde Q(x,c) \leq \frac 1 2 \tilde Q_{xx}(x,c) + \lambda x^2, \quad x \in [0,\infty) \setminus \{f(c),g(c),f_0+c\}, c \in [0,c_0); \label{eq:45}
\end{gather}
and
\begin{gather}\label{eq:46}
[\d x^2 - \tilde Q(x,c)]\cdot[1-\tilde Q_x(x,c)-\tilde Q_c(x,c)]\cdot\left[\frac 1 2 \tilde Q_{xx}(x,c)+\lambda x^2-\a \tilde Q(x,c)\right]=0
\end{gather}
for $x \in [0,\infty) \setminus \{f(c),g(c),f_0+c\}$, $c \in (0,c_0)$. Then $Q(x,c)\geq \tilde Q(x,c)$ for all $x \geq 0$, $c \in [0,c_0)$. 
\end{theorem}

\begin{Remark}
Theorem \ref{th:ver} is a modification of Theorem 4.1 in \cite{KOWZ00}. The modifications are only to remove the requirement that $\tilde Q$ be twice differentiable across the line $x=f_0+c$, and to restrict the domain of $\tilde Q$ to $\R \times [0,c_0)$. The second derivative in $x$ across $x=f_0+c$ can be dealt with in exactly the same way as it is across $x=f(c)$ and $x=g(c)$ in the proof of the latter theorem, and the restriction to monotone controls (Remark \ref{rem:halfspace}) ensures that the latter proof is also valid for $\tilde Q$ defined on this restricted domain. \end{Remark}

Recalling Proposition \ref{existencegeometric} and the definition of $U$ from \eqref{eq:Vstar}-\eqref{eq:UUdef}, from now on we set $A^* \equiv A$ and $B^* \equiv B$. We begin with the following lemma.

\begin{Lemma} \label{lem:atleastone}
For each $c \in (0,c_0)$, the function $x \mapsto U(x,c)$ has at most one turning point on $(F(c),G(c))$.
\end{Lemma}
\begin{proof} Let $c \in (0,c_0)$.
From part 2 of Proposition \ref{existencegeometric}, rewriting the common tangent as
\begin{equation}
W(y;c) = H(\hyt(c);c) + (y - \hyt(c))H_y(\hyt(c);c), 
\end{equation}
we obtain 
\begin{equation}
W_c(y;c) = H_c (\hyt(c);c) + (y - \hyt(c))\lb H_{yy}(\hyt(c);c)\hyt'(c) + H_{yc}(\hyt(c);c)\rb
\end{equation}
for $y \in (\hyo(c), \hyt(c))$. It is therefore clear that $c \mapsto W(y;c)$ (and hence also $x \mapsto \tilde V(x;c)$, by Proposition \ref{prop:DayKar}) is continuously differentiable across the right-hand moving boundary, and also across the left-hand moving boundary (the $c$ derivatives vanish here). Moreover $W_c$ is uniformly continuous in an open neighbourhood of $(\hyt(c),c)$, so we conclude that $W$ is continuously differentiable in this open neighbourhood; and similarly at $(\hyo(c),c)$. Differentiating \eqref{eq:Qdef} we therefore obtain
\begin{eqnarray}
U(x,c) &=& \ftla x + C(c)e^{x\sta} + D(c)e^{-x\sta}, \label{eq:Uexp} \\
U(G(c),c) &=& 1, \quad U (F(c),c) = 2 \d F(c) \in (0,1), \label{eq:Ubdry}
\end{eqnarray}
for some functions $C$ and $D$.
For fixed $c$, note that $e^{x\sta} U_x(x,c)$ is quadratic in the variable $X=e^{x\sta}>0$. In order for $U_x(\cdot,c)$ to have two zeroes, both of the following roots must therefore be strictly positive:
\[
X_\pm: = \frac 1 {C(c)}\left(-\frac{\la}{\a\sta} \pm \sqrt{\frac{ \la^2}{2 \a^3} + C(c) D(c)}
\right).
\]
This in turn can occur only if both $C(c)$ and $C(c)D(c)$ are strictly negative, in which case \eqref{eq:Uexp} gives that $U_x(\cdot,c)$ is strictly concave. In this case, note from 
\eqref{eq:C15} that $U_x(G(c);c)$ and $q(G(c);F(c))$ have the same sign. It then follows from the definition of $c_0$ in  \eqref{lem:Gderv} that
\begin{equation}\label{eq:Uxpos}
U_x(G(c);c) > 0, \qquad c \in (0,c_0),
\end{equation}
and the required result follows.
\end{proof}\hfill$\Box$

\begin{Theorem} Define $\tilde V(\cdot;0)= \tilde V_0(\cdot)$. Then the even (in $x$) extension of the function $\tilde V$  to $\R \times [0,c_0)$, the constant $c_0$, and the moving boundaries $F$ and $G$ (with $F(0)=G(0)=\fotd$)
 satisfy the conditions of Theorem \ref{th:ver}.
\end{Theorem}

\begin{proof}
We first note that this extension of $\tilde V$ is smooth across the boundary $x=0$ (where it is locally a quadratic in $x$). We now proceed to establish the regularity of the function $(x,c) \mapsto \tilde V(x;c)$.

It should now be clear from Figure \ref{fig:Fig1synth} that for each $x \geq 0$ we have
\begin{equation}
\lim_{c \to 0}V(x;c) = V(x;0) = V_0(x). \label{eq:cont2}
\end{equation}

This can be proved as follows: as $c \to 0$, $H_r(\cdot;c)$ converges pointwise to the minorant $W_0(\cdot)$ shown in Figure \ref{fig:1}  (this follows from the proof of Lemma \ref{lem:minorant}).
It is then easy to see 
that $H(\cdot;c):=H_l(\cdot) \wedge H_r(\cdot;c)$ also converges pointwise to $W_0(\cdot)$. Since the free boundaries (Figure \ref{fig:Fig1synth}, filled circles) converge to their common limit (filled square) by Lemma \ref{cor:1/2d0}, it follows that the minorant $W(\cdot;c)$ of $H(\cdot;c)$ also converges pointwise to $W_0(\cdot)$. Equation \eqref{eq:cont2} now follows by the usual transformation (Proposition \ref{prop:DayKar}).

The fact that $\tilde V \in  C^1(\R_+ \times (0,c_0))$ was established in the proof of Lemma \ref{lem:atleastone}. The existence (almost everywhere) and local boundedness of the $x$ derivatives, the growth condition \eqref{eq:42}, and the equality \eqref{eq:46} follow from the explicit expressions for $\tilde V(x;c)$ obtained above. 
The bound \eqref{eq:43} also follows immediately from the definition of problems \eqref{def-V}-\eqref{eq:hldef}.
For condition \eqref{eq:44}, the positivity of $U_x(G(c);c)$ from \eqref{eq:Uxpos} and
Lemma \ref{lem:atleastone} (cf. \eqref{eq:Ubdry}) ensure that $U(x,c) \leq 1$ for $x \in (F(c),G(c))$ and condition \eqref{eq:44} follows.

For \eqref{eq:45} we have 
\begin{eqnarray}
R( x, c)&:=&\frac 1 2 \tilde V_{xx}( x; c)+\la  x^2-\a \tilde V( x; c) \nonumber \\
&=& \left\{\begin{array}{ll} \d - (\a \d - \la)x^2 > \d - (\a \d - \la)f_0^2 > 0, & x \in [0,F(c)), \\
0, & x \in (F(c),G(c)), \\
 \d + \la  x^2 - \a(\d( x- c)^2+ c),&  x \in (G(c),f_0+c),
\\
2 \la c \lb x - \lb \frac c 2 + \atl \rb
\rb \geq 0, & x > f_0+c,
\end{array}\right. \label{eq:cond69}
\end{eqnarray}
and so when $c \in (0,c_0)$ (and hence $\fotd+ c<G( c)$) we have
\begin{eqnarray}
R_c( x, c)
&=& \left\{\begin{array}{ll} 2 \la  \lb x - \lb c + \atl \rb
\rb >0, & x > f_0+c,
\\ \a(2\d( x- c)-1) > 0,&  x \in (G(c),f_0+c).
\end{array}\right. \label{eq:dRdc}
\end{eqnarray} 
The nonnegativity of $R(\cdot,0)$ follows as in \eqref{eq:cond69}.
Then from \eqref{eq:dRdc} 
we have
\[
R( x,  c) = R( x,0) + \int_0^{ c} R_c( x,u)du \geq 0,
\]
as required. 
\end{proof}
\hfill$\Box$

\section{Acknowledgements}
The author was supported by EPSRC grant EP/K00557X/1 and would like to thank Tiziano De Angelis and Goran Peskir for several useful discussions and comments which significantly improved an early draft of the paper. 
\begin{appendix}

\section{Derivation of \eqref{eq:lx} and \eqref{eq:hall3}}
\label{app:hall}
We include the following derivations for completeness.
Differentiation of \eqref{eq:th3} and \eqref{eq:th4} yields
$$
e^{x\sta}\pd {\tilde \cH_3} x (x,c)= - e^{-x\sta}\pd {\tilde \cH_4} x (x,c),
$$
and writing $H$ for $h_2 \circ h_1^{-1}$ we have $H'(w)=-e^{2h_1^{-1}(w)\sta}$. 
Differentiating \eqref{eq:Ldot} then yields \eqref{eq:lx} as required. A further differentiation and the use of \eqref{eq:hall1}-\eqref{eq:hall2} also gives
\begin{equation}\label{eq:fonz}
(L_x+L_c)(x,c)= \sta \lb \tilde \cH_4 (x,c) - \tilde \cH_3 (x,c) e^{2z \sta} \rb + h_3(x) e^{2z \sta} - h_4(x),
\end{equation}
where $z=h_1^{-1}(\tilde \cH_3(x,c))$.
For $c \in (0,c_2)$ we have $\tilde \cH_3(G(c),c)=h_1(F(c))$ and $\tilde \cH_4(G(c),c)=h_2(F(c))$. Evaluating \eqref{eq:fonz} at $x=G(c)$, the required identity then follows from \eqref{eq:qdef}.

\end{appendix}

\end{document}